\documentclass[journal]{IEEEtran}
\usepackage{enumerate}
\usepackage{cite}
\usepackage[citecolor=black]{hyperref}
\usepackage{lineno}
\usepackage[cmex10]{amsmath}
\usepackage{amscd,amsfonts,amsmath,amssymb,mathrsfs,pifont,stmaryrd,tipa}
\usepackage{array,cases,dsfont,graphicx,texdraw}
\usepackage{wrapfig}
\usepackage{graphics} % for pdf, bitmapped graphics files
\usepackage{epsfig} % for postscript graphics files
\usepackage{stfloats}
\usepackage{subfig}
\usepackage{hyperref}
\hypersetup{
	colorlinks=true,
	linkcolor=black
}
\usepackage{color}
\newtheorem{Remark}{Remark}
\newtheorem{Corollary}{Corollary}
\newtheorem{Definition}{Definition}

\newenvironment{Proof}{\noindent{\em Proof:\/}}{\hfill $\Box$\par}
\newtheorem{Theorem}{Theorem}
\newtheorem{Lemma}{Lemma}

\newtheorem{Assumption}{Assumption}

\input{tcilatex}                                                          % paper

\IEEEoverridecommandlockouts                              % This command is only
\title{\LARGE \bf
	Fully Distributed Nash Equilibrium Seeking for High-order Players with Bounded Controls and Directed Graphs}

\author{Maojiao Ye, Lei Ding,  Shengyuan Xu
	\thanks{M. Ye and S. Xu are with the School of Automation, Nanjing University of Science and Technology, Nanjing 210094, P.R. China (Email: ye0003ao@e.ntu.edu.sg; syxu@njust.edu.cn); L. Ding is with Institute of Advanced Technology, Nanjing University of Posts and Telecommunications (Email: dl522@163.com).}
	\thanks{This work is supported by the National Natural Science Foundation of China (NSFC), No. 61803202, 62073171, the Natural Science Foundation of
		Jiangsu Province, No. BK20180455, BK20200744, and the Fundamental Research Funds for the Central Universities, No. 30920032203.}
}
\begin{document}
	
	\maketitle
	\thispagestyle{empty}
	\pagestyle{empty}
	
	\begin{abstract}
		This paper explores  distributed Nash equilibrium seeking problems for games in which 
		the players have limited knowledge on  other players' actions. In particular, the involved players are considered to be high-order integrators with their control inputs constrained within a pre-specified region.
	 A linear transformation for the players' dynamics is firstly utilized to facilitate the design of bounded control inputs incorporating multiple saturation functions.
%		 A linear transformation is firstly utilized to convert the players' dynamics, based on which the control inputs are designed with multiple saturation functions utilized to ensure the boundedness of the control inputs.
		 By introducing consensus protocols with adaptive and time-varying gains, the unknown actions for players are distributively estimated. Then, a fully distributed Nash equilibrium seeking strategy is exploited, showcasing its remarkable properties:  i) ensuring the boundedness of control inputs; ii) avoiding any global information/parameters; and iii) allowing the graph to be directed. 
		Based on Lyapunov stability analysis, it is theoretically proved that the proposed distributed control strategy can lead all the players' actions to the Nash equilibrium. Finally, an illustrative example is given to validate effectiveness of the proposed method.
	\end{abstract}


%%%%%%%%%%%%%%%%%%%%%%%%%%%%%%%%%%%%%%%%%%%%%%%%%%%%%%%%%%%%%%%%%%%%%%%%%%%%%%%%
\section{Introduction}

As a fundamental and key issue to be addressed for game theoretical applications to large-scale distributed systems, Nash equilibrium seeking in neighboring-communication environment has attracted remarkable attention in the past several years from researchers in the control community \cite{YETAC2017}--\cite{SalehisadaghianiAuto16}. For  practical control engineering problems, communication structure (i.e., undirected or directed), system dynamics  and actuator limitations are all critical factors that may seriously influence control design and implementation.  In this regard, to promote the penetration of game theoretical approaches for distributed control systems, it is essential to develop distributed Nash equilibrium seeking strategies taking these factors into full consideration. In spirit of broadening the applicable fields of distributed games, some efforts have been made to deal with high-order players, e.g., see  \cite{ZhengCCC2020, LiuICCA2020, RomanoTCSN20}.
 However, for games with high-order players, there have been few works regarding actuator limitations and fully distributed designs under directed graphs.

It is well recognized that,  due to
hard physical constraints, it is inevitable for players to suffer from the limitation of control inputs/actuation in practical distributed game applications, which probably causes degradation or even damage of control performance. In order to address this issue, \cite{YETAC2021} constructed bounded controls for first- and second-order integrator-type systems to find the Nash equilibrium. Moreover, high-order players were considered in \cite{AiIJRNC2021}. Backstepping techniques were employed for the control design and the ``explosion of terms" induced by backstepping was addressed through a fixed-time sliding mode observer. However, the seeking strategies \cite{YETAC2021,AiIJRNC2021} contain centralized control gains whose explicit quantifications require the knowledge on the network structure, the size of the game as well as the players' objective functions. 

As centralized information can hardly be obtained by every engaged player  in practical situations, the tuning of control gains is in fact a matter of trial and error. In particular, when communication structures change or there is any player joining/leaving the game \cite{YEAUTO2021}, the control gains for the designed strategies may need to be re-quantified, which implies the loss of plug and play property. To address this problem,   \cite{YEAUTO2021,Persis2019,Bianchi2021} found some ways out by proposing adaptive designs for control gains and thus developed fully distributed control laws for games in neighboring communication environment. Different from the two-hop communication based algorithms constructed in \cite{Persis2019,Bianchi2021}, only one-hop communication is needed in \cite{YEAUTO2021}. However, it should be pointed out that the adaptive designs in \cite{YEAUTO2021,Persis2019,Bianchi2021} is only applicable for undirected graphs and cannot be directly extended to deal with directed graphs.  To the best of the authors' knowledge, {how to achieve fully distributed Nash equilibrium seeking under directed graphs is still an open and challenging  issue}. Furthermore, it is noted that, practical actuator limitations introduce high nonlinearity and bring some difficulties in distributed control design for games, but they are not taken into account  in \cite{YEAUTO2021,Persis2019,Bianchi2021}. Therefore, it is a non-trivial and challenging task  to establish Nash equilibrium seeking strategies under  bounded controls in a fully distributed manner, especially when  communication topologies are directed. 

Motivated by the above observations, this paper aims to develop fully distributed control laws for high-order players, which can accommodate actuator  limitations and directed communication structures. Highlighting the improvements for the existing works, the contributions and novelties of this paper are stated as follows. 
\begin{enumerate}[i)]
	\item   This paper solves games with high-order players whose control inputs are required to bounded in a fully distributed fashion. By employing a linear transformation to convert the players' dynamics, the control inputs with  multiple saturation functions are constructed.  Through a synthesis of an optimization method, a consensus algorithm and time-varying/adaptive gain designs, a fully distributed Nash equilibrium seeking strategy with bounded control inputs is established.  
	\item As first- and second-order dynamics are special cases of high-order ones, the presented methods  provide alternative approaches for the problem considered in \cite{YETAC2021}, while covering more general cases and removing the requirements on any global information.   In addition, the presented methods can accommodate the heterogeneity on the system order and require less computation expenditure than that of \cite{AiIJRNC2021}, especially when the order of the system is high.
	\item   The proposed strategy is fully distributed in the sense that no centralized control gains are involved and no knowledge on any global information is required for  the players. In particular, compared with the adaptive designs under undirected graphs in \cite{YEAUTO2021,Persis2019,Bianchi2021}, the proposed strategy allows the graph to be directed and only requires one-hop communication, which is preferable for distributed systems.
	\item   The proposed methods are analytically studied and it is theoretically proven that the Nash equilibrium is globally asymptotically stable under the proposed methods.  Several special cases are discussed to provide more insights on the connections with the existing works.   
\end{enumerate}

%The remaining parts of the paper are given in the following order. Section \ref{NP} presents the formulated problem, and Section 
%\ref{main} proposes the control laws to achieve the goal of the paper. The context in Section \ref{dis_cu} focuses on some discussions on the presented results. Numerical verifications are given in Section \ref{num} and conclusions are presented in Section \ref{concl}.

\section{Problem Statement}\label{NP}
This paper considers a network of high-order integrator-type players with labels from $1$ to $N,$ successively, where $N>1$ is an integer. The state of player $i$, denoted as $x_i\in\mathbb{R}^{m_i}$, is generated by 
\begin{align}\label{eq_integra}
	\dot{x}_i=A_ix_i+B_iu_i,~~~~y_i=C_ix_i,\nonumber
\end{align}
where $A_i=\begin{bmatrix}
\mathbf{0}_{m_i-1} &I_{m_i-1}\\
0 &\mathbf{0}_{m_i-1}^T
\end{bmatrix}\in\mathbb{R}^{m_i\times m_i}$, $B_i=\begin{bmatrix}
0&\cdots& 0&1
\end{bmatrix}^T\in\mathbb{R}^{m_i\times 1},$ $C_i=\begin{bmatrix}
1& 0& \cdots& 0 &0
\end{bmatrix}\in\mathbb{R}^{1\times m_i}$ and $m_i>1$ is a positive integer.  Moreover, $u_i\in\mathbb{R}$ and $y_i\in\mathbb{R}$ are the control input and output/action of player $i$, respectively.  Assume that each player has a local objective function defined as $f_i(\mathbf{y})=f_i(y_i,\mathbf{y}_{-i}),$ where $\mathbf{y}=[y_1,y_2,\cdots,y_N]^T,$ $\mathbf{y}_{-i}=[y_1,\cdots,y_{i-1},y_{i+1},\cdots,y_N]^T,$ and each player aims at minimizing  $f_i(y_i,\mathbf{y}_{-i})$ through adjusting its own action $y_i$, i.e.,
\begin{align}
	\text{min}_{y_i} \ \ \  f_i(y_i,\mathbf{y}_{-i}). 
\end{align}

Suppose that each player cannot directly access all other players' actions and
\begin{align}\label{bound}
	|u_i|\leq U_i
\end{align}
where $U_i$ is a positive constant denoting the actuator limitation of player $i$. 

The paper aims to design \textbf{fully} distributed $u_i$ that satisfies \eqref{bound} to drive the players' actions $\mathbf{y}$ to the Nash equilibrium $\mathbf{y}^*$, whose definition is given below.  
\begin{Definition}
An action profile $\mathbf{y}^*=(y_i^*,\mathbf{y}_{-i}^*)$ is a Nash equilibrium if
for all $y_i\in\mathbb{R},i\in\mathcal{V},$
\begin{equation}
f_i(y_i^*,\mathbf{y}_{-i}^*)\leq f_i(y_i,\mathbf{y}_{-i}^*),
\end{equation}
where $\mathcal{V}$ is the player set given as $\mathcal{V}=\{1,2,\cdots,N\}.$
\end{Definition}
\begin{Remark}
	It is worth mentioning that in the paper $x_i\in \mathbb{R}^{m_i}$, where $m_i$ for $i\in\mathcal{V}$ can be different from each other, indicating that the \textit{heterogeneity} on the order of the players' dynamics is allowed.
\end{Remark}

For notational clarity, define $[\chi_i]_{vec}$   as a column vector whose $i$th entry is $\chi_i$. Moreover, let $[\chi_{ij}]_{vec}$ ($\text{diag}\{\chi_{ij}\}$) for $i,j\in\mathcal{V}$ be a column vector (diagonal matrix) whose entries are $\chi_{11},\chi_{12},\cdots,\chi_{1N},\chi_{21},\cdots,\chi_{NN}$, respectively. In addition, $[\chi_{ij}]$ is a matrix whose $(i,j)$th entry is $\chi_{ij}.$

The remaining sections are based on the assumptions below. 

\begin{Assumption}\label{ass1}
The players' objective functions $f_i(\mathbf{y})$ for $i\in\mathcal{V}$ are continuously differentiable with their gradients $\nabla_i f_i(\mathbf{y})$ being globally Lipschitz, i.e., for $\mathbf{y},\mathbf{z}\in\mathbb{R}^{N},$
\begin{equation}
\left|\left|\nabla_i f_i(\mathbf{y})-\nabla_i f_i(\mathbf{z})\right|\right|\leq l_i||\mathbf{y}-\mathbf{z}||, \forall i\in\mathcal{V},
\end{equation}
for some positive constant $l_i,$ where $\nabla_i f_i(\mathbf{y})=\frac{\partial f_i(\mathbf{y})}{\partial y_i}$ and  
$\nabla_i f_i(\mathbf{z})=\frac{\partial f_i(\mathbf{y})}{\partial y_i}\left.\right|_{\mathbf{y}=\mathbf{z}}.$
\end{Assumption}

\begin{Assumption}\label{ass2}
For  $\mathbf{y},\mathbf{z}\in\mathbb{R}^{N},$ 
\begin{equation}
(\mathbf{y}-\mathbf{z})^T([\nabla_i f_i(\mathbf{y})]_{vec}-[\nabla_i f_i(\mathbf{z})]_{vec})\geq \omega||\mathbf{y}-\mathbf{z}||^2,
\end{equation}
for some positive constant $\omega.$ 
\end{Assumption}

To design fully distributed control laws, it is assumed that there is a directed communication graph among the players described by $\mathcal{G}(\mathcal{V}, \mathcal{E})$, where $\mathcal{V}$ and $\mathcal{E}\subseteq\mathcal{V}\times\mathcal{V}$ stand for  the node set and edge set, respectively. Let $(i, j)\in \mathcal{E}$ and $a_{ji}$ be an edge from node $i$ to $j$ and its weight, respectively. If $(i, j)\in \mathcal{E}$, $a_{ji}>0$, otherwise, $a_{ji}=0$. In this paper,  $a_{ii}=0$. 
A directed path is defined as a sequence of edges of the form $(i_1,i_2),$ $(i_2,i_3),\cdots.$ A directed graph is strongly connected if for every pair of distinct nodes, there is a path. 
Define $\mathcal{A}=[a_{ij}]$ as the adjacency matrix of $\mathcal{G}$. Then, $\mathcal{L}=\mathcal{D}-\mathcal{A}$, where $\mathcal{D}=\text{diag}\{\sum_{j=1}^{N}a_{ij}\},$ is the Laplacian matrix of $\mathcal{G}$ \cite{QinTAC}--\cite{ShaoTAC20}. 

\begin{Assumption}\label{ass3}
	The directed graph $\mathcal{G}$ is strongly connected.
\end{Assumption}

\begin{Remark}
Assumptions \ref{ass1}-\ref{ass3} are commonly adopted and mild conditions (see, e.g., \cite{YETAC2017},\cite{wangauto2020},\cite{YETAC183} and many other references therein). Assumption \ref{ass2} is employed to characterize a unique Nash equilibrium, which is globally exponentially stable under the gradient play for games with globally Lipschitz gradients (Assumption \ref{ass1}) \cite{YETAC2017}. While existing fully distributed Nash equilibrium seeking schemes \cite{YEAUTO2021},\cite{Persis2019},\cite{Bianchi2021} are established for undirected communication topologies, Assumption \ref{ass3} suggests that asymmetric information exchange  among the players is sufficient for the developed methods.   
\end{Remark}

\section{Main Results}\label{main}
This section develops a fully distributed Nash equilibrium seeking strategy for the considered problem, under which the associated convergence analysis is provided. 
\subsection{Strategy Design}

To deal with the players' dynamics, a transformation is firstly conducted by defining $x_i=T_i\bar{x}_i$ to convert \eqref{eq_integra} to
\begin{align}\label{eq_transf}
	\dot{\bar{x}}_i=\bar{A}_i\bar{x}_i+\bar{B}_iu_i,
\end{align}
in which $\bar{A}_i=\left[\begin{smallmatrix}
0& \theta_i^{m_i-1}& \theta_i^{m_i-2}& \cdots&\theta_i\\
0 &0 & \theta_i^{m_i-2}&\cdots&\theta_i\\
0&0 & 0&\cdots&\theta_i\\
\vdots&\vdots &\vdots &\ddots& \vdots\\
0 &0 &0 &\cdots &0
\end{smallmatrix}\right]$ and $\bar{B}_i=\begin{bmatrix}
1 &
1&
\cdots&
1
\end{bmatrix}^T,$
 $T_i=\mathbf{R}(A_i,B_i)\mathbf{R}(\bar{A}_i,\bar{B}_i)^{-1}$ is a non-singular matrix , $\mathbf{R}(A,B)$ denotes the controllability matrix of $(A,B)$ and $\theta_i\in(0,1)$ is a constant to be further determined \cite{SussmannTAC94}.  
Based on the above transformation, the fully distributed bounded control input $u_i$ is designed as 
\begin{align}\label{eq1}
	u_i=&-\sum\nolimits_{k=1}^{m_i-1}\theta_i^k \phi_i(\bar{x}_{i(m_i-k+1)})\nonumber\\
	&-\theta_i^{m_i}\phi_i(\bar{x}_{i1}+\prod_{k=1}^{m_i-1}\theta_i^k\int_{0}^{t}\nabla_{i}f_i(\mathbf{z}_i(\tau))d\tau),
	\end{align}
in which  $\mathbf{z}_i=[z_{i1},z_{i2},\cdots,z_{iN}]^T$ and 
\begin{align}\label{ada}
	\dot{z}_{ij}=&-(c_{ij}+\rho_{ij})(\sum\nolimits_{k=1}^N a_{ik}(z_{ij}-z_{kj})\nonumber\\
	&+a_{ij}(z_{ij}+\int_{0}^{t}\nabla_{j}f_j(\mathbf{z}_j(\tau))d\tau)),\nonumber\\
	\dot{c}_{ij}=&\rho_{ij},
\end{align}
 $\rho_{ij}=(\sum_{k=1}^N a_{ik}(z_{ij}-z_{kj})+a_{ij}(z_{ij}+\int_{0}^{t}\nabla_{j}f_j(\mathbf{z}_j(\tau))d\tau))^2,$
  and  
  $c_{ij}(0)>0.$ 
Moreover, $\phi_i(\cdot)$ is a saturation function defined as $\phi_i(\varsigma)=sign(\varsigma)\min\{|\varsigma|,\Delta_i\},$ where $\Delta_i$ is a positive constant that can be adjusted according to the actuator limitation. 

\begin{Remark}
	The saturation function utilized in the control design ensures the boundedness of the control inputs. More specifically, given any positive constant $U_i,$ one can choose $\Delta_i$ such that 
	\begin{align}
		\frac{\theta_i}{1-\theta_i}\Delta_i<U_i,
	\end{align}
	to ensure that $|u_i|\leq U_i$.  
\end{Remark}
\begin{Remark}
The adaptive design, inspired by \cite{WangTCSII29}, ensures that $\rho_{ij}(t)$ is non-negative and  $c_{ij}(t)$ is monotonically increasing as $\dot{c}_{ij}(t)$ is non-negative. Moreover, $\theta_i$ can be determined in a decentralized fashion. Therefore, all the control gains are independent of any global information. In addition, the update of the auxiliary variables $z_{ij}$ depends only on local information exchange. Hence, the control input in \eqref{eq1}-\eqref{ada} is \textit{fully} distributed. Note that as the communication graph is directed in this paper, the adaptive designs in \cite{YEAUTO2021,Persis2019,Bianchi2021} cannot be applied.
\end{Remark}
\begin{Remark}
It is worth mentioning that the linear transformation is not unique. For example,  one can also choose a non-singular matrix $T_i$ to convert the players' dynamics to 
\begin{align}\label{eq_transf11}
\dot{\hat{x}}_i=\hat{A}_i\hat{x}_i+\hat{B}_iu_i,
\end{align}
in which $\hat{A}_i=\left[\begin{matrix}
0& \theta_i& \theta_i& \cdots&\theta_i\\
0 &0 & \theta_i&\cdots&\theta_i\\
0&0 & 0&\cdots&\theta_i\\
\vdots&\vdots &\vdots &\ddots& \vdots\\
0 &0 &0 &\cdots &0
\end{matrix}\right]$ and $\hat{B}_i=\begin{bmatrix}
1 &
1&
\cdots&
1
\end{bmatrix}^T.$
In the case, $u_i$ can be designed as 
\begin{align}\label{eqzz1}
u_i=&-\sum\nolimits_{k=1}^{m_i-1}\theta_i \phi_i(\hat{x}_{i(m_i-k+1)})\nonumber\\
&-\theta_i\phi_i(\hat{x}_{i1}+\theta_i^{m_i-1}\int_{0}^{t}\nabla_{i}f_i(\mathbf{z}_i(\tau))d\tau),
\end{align}
for which the convergence analysis follows that of \eqref{eq1}-\eqref{ada}.
\end{Remark}

\subsection{Convergence Analysis}

In this section, the method in  \eqref{eq1}-\eqref{ada} is analytically investigated. 
Before proceeding to the convergence analysis, the following supportive lemmas are given. 

\begin{Lemma}\label{Lemma1}
For each $\theta_i\in (0,\frac{1}{2})$, there exists a constant $T(\theta_i)\geq 0$ such that for all $i\in\mathcal{V},$
	\begin{equation}
		|\bar{x}_{ik}(t)|\leq \Delta_i, \forall t\geq T,k\in\{2,\cdots,m_i\}.
	\end{equation}
\end{Lemma}
\begin{Proof}
	See Section \ref{proof_Lemma1} for the proof.
\end{Proof}
\begin{Remark}
	Lemma \ref{Lemma1} demonstrates that  there exists a non-negative constant $T$ such that for $t\geq T,$ $|\bar{x}_{ik}(t)|$ for all  $k=2,\cdots,m_i-1,m_i$ can evolve into the unsaturated region, indicating that the effects of the saturation function on $|\bar{x}_{ik}(t)|$ for $k=2,\cdots,m_i-1,m_i$ vanish within finite time. Based on this conclusion, the stability analysis for the closed-loop system is largely simplified. 
\end{Remark}

Now, we focus on the evolution of $\bar{x}_{i1}(t)$ by considering a reduced system as
\begin{align}
\dot{\bar{x}}_{i1}=-\theta_i^{m_i}\phi_i(\bar{x}_{i1}+\prod_{k=1}^{m_i-1}\theta_i^k\int_{0}^{t}\nabla_{i}f_i(\mathbf{z}_i(\tau))d\tau)\label{auxi_2}.
\end{align}
 
Let $\tilde{x}_{i1}=\bar{x}_{i1}+\prod_{k=1}^{m_i-1}\theta_i^k\int_{0}^{t}\nabla_{i}f_i(\mathbf{z}_i(\tau))d\tau.$ Then,
\begin{align}\label{auxi_3}
\dot{\tilde{x}}_{i1}
%&=\dot{\bar{x}}_{i1}+\prod_{k=1}^{m_i-1}\theta_i^k\nabla_{i}f_i(\mathbf{z}_i(t))\nonumber\\
&=-\theta_i^{m_i}\phi_i(\tilde{x}_{i1})+\prod\nolimits_{k=1}^{m_i-1}\theta_i^k\nabla_{i}f_i(\mathbf{z}_i(t)).
\end{align}
Consequently, the subsequent result can be derived. 
\begin{Lemma}\label{lemma2}
	Suppose that $|\nabla_{i}f_i(\mathbf{z}_i(t))|\leq \nu_1$ for all $t>0$ and there is a constant $\tilde{T}\geq 0$ such that for all $t\geq\tilde{T},$
	$\Delta_i>\frac{2}{\theta_i^{m_i}}|\prod_{k=1}^{m_i-1}\theta_i^k\nabla_{i}f_i(\mathbf{z}_i(t))|.$
	Then, the trajectory $\tilde{x}_{i1}(t)$ generated by \eqref{auxi_3} stays bounded for all $t>0$ and there exists a $\beta\in \mathcal{KL}$ and a $\gamma\in \mathcal{K}$ such that for $t\geq\tilde{T},$
		\begin{align}
		|\tilde{x}_{i1}(t)|\leq & \beta(|\tilde{x}_{i1}(\tilde{T})|,t-\tilde{T})\nonumber+\gamma(\sup_{\tilde{T}<\tau<t} |\nabla_{i}f_i(\mathbf{z}_i(\tau))|).
		\end{align}
\end{Lemma}

\begin{Proof}
	See Section \ref{proof_lemma2} for the proof. 
\end{Proof}

	Lemma \ref{lemma2} demonstrates that with bounded $\nabla_{i}f_i(\mathbf{z}_i(t))$, the trajectory of $\tilde{x}_{i1}(t)$ will always stay bounded. In addition, if $|\nabla_{i}f_i(\mathbf{z}_i(t))|$ is decreasing to be sufficiently small and stays therein thereafter, $|\tilde{x}_{i1}(t)|$ will be upper-bounded by  $\beta(|\tilde{x}_{i1}(\tilde{T})|,t-\tilde{T})+\gamma(\sup_{\tilde{T}<\tau<t} |\nabla_{i}f_i(\mathbf{z}_i(\tau))|),$ indicating that if $|\nabla_{i}f_i(\mathbf{z}_i(t))|$ vanishes to zero as $t\rightarrow\infty$, 
	\begin{align}
		\lim_{t\rightarrow \infty} |\bar{x}_{i1}(t)+\int_{0}^{t}\prod_{k=1}^{m_i-1}\theta_i^k\nabla_{i}f_i(\mathbf{z}_i(\tau))d\tau|=0.
	\end{align}
To this end, one needs to further consider the evolution of $\nabla_{i}f_i(\mathbf{z}_i(t))$, which is investigated by considering the following auxiliary system,
\begin{align}
\dot{z}_{ij}=-(c_{ij}+\rho_{ij})\xi_{ij},~~~~~~\dot{c}_{ij}=\rho_{ij}.\label{auxi_1}
\end{align}
where $\rho_{ij}=\xi_{ij}^2$, $c_{ij}(0)>0,$ $\xi_{ij}=\sum_{k=1}^N a_{ik}(z_{ij}-z_{kj})+a_{ij}(z_{ij}+\int_{0}^{t}\nabla_{j}f_j(\mathbf{z}_j(\tau))d\tau).$ Let $\mathbf{\xi}=[\xi_{ij}]_{vec}$, $\mathbf{z}=[z_{ij}]_{vec}$, $H=\mathcal{L}\otimes I_{N}+A_0,$ $A_0=\text{diag}\{a_{ij}\}$, $c=\text{diag}\{c_{ij}\}$ and $\rho=\text{diag}\{\rho_{ij}\}$. Then, $\mathbf{\xi}=H(\mathbf{z}+\mathbf{1}_N\otimes [\int_{0}^{t}\nabla_{i}f_i(\mathbf{z}_i(\tau))d\tau]_{vec}),$
%\begin{align}
%\mathbf{\xi}=H(\mathbf{z}+\mathbf{1}_N\otimes [\int_{0}^{t}\nabla_{i}f_i(\mathbf{z}_i(\tau))d\tau]_{vec}),
%\end{align}
and $\dot{\mathbf{\xi}}=H(-(c+\rho)\mathbf{\xi}+\mathbf{1}_N\otimes \nabla_{i}f_i(\mathbf{z}_i(t))).$
%\begin{align}
%\dot{\mathbf{\xi}}=H(-(c+\rho)\mathbf{\xi}+\mathbf{1}_N\otimes \nabla_{i}f_i(\mathbf{z}_i(t))).
%\end{align}

The following result can be obtained. 

\begin{Lemma}\label{lemma3}
	Under Assumptions \ref{ass1}-\ref{ass3}, 
	\begin{align}
		&\lim_{t\rightarrow \infty}||-[\int_{0}^{t}\nabla_{i}f_i(\mathbf{z}_i(\tau))d\tau]_{vec}-\mathbf{y}^*||=0,\nonumber\\
		&\lim_{t\rightarrow \infty} ||\mathbf{z}(t)+\mathbf{1}_N\otimes [\int_{0}^{t}\nabla_{i}f_i(\mathbf{z}_i(\tau))d\tau]_{vec}||=0.
	\end{align}
%	and 
%	\begin{align}
%		\lim_{t\rightarrow \infty} ||\mathbf{z}(t)+\mathbf{1}_N\otimes [\int_{0}^{t}\nabla_{i}f_i(\mathbf{z}_i(\tau))d\tau]_{vec}||=0.
%	\end{align}
	Moreover, $c_{ij}$ for all $i,j\in\mathcal{V}$  converge to some finite values. 
\end{Lemma}
\begin{Proof}
	See Section \ref{proof_lemma3} for the proof.
\end{Proof}

Based on the above results, the convergence result can be established for the control design in \eqref{eq1}.

\begin{Theorem}\label{them1}
	Under Assumptions \ref{ass1}-\ref{ass3} and the control input in \eqref{eq1}, 
	\begin{align}
	\lim_{t\rightarrow \infty} ||\mathbf{y}(t)-\mathbf{y}^*||=0.   
	\end{align} 
	In addition, all the other variables stay bounded and converge to finite values. 
\end{Theorem}
\begin{proof}
	See Section \ref{proof_them1} for the proof. 
\end{proof}

Theorem \ref{them1} illustrates that the Nash equilibrium is \textit{globally} asymptotically stable though the boundedness of the control inputs is considered. Furthermore, all the other variables (i.e., $\mathbf{x}_{i}(t)$, $c_{ij}(t)$ and $z_{ij}(t)$ for all $i,j\in\mathcal{V}$) stay bounded and converge to some finite values.

\section{Discussions on The Presented Results}\label{dis_cu}
In this section, we provide some insights on the presented results, in terms of first- and second- order players, undirected graph and no actuator limitation. This helps to establish a link between the presented results and the existing works.

\subsection{First- and second-order integrator-type players}

%In Section \ref{main}, it is supposed that $m_i>1$. However, if $m_i=1,$ i.e., 

For $m_i=1,$ the system \eqref{eq_integra} can be written as
\begin{align}
	\dot{x}_{i1}=u_i, \ \ \ \ \ \ \ y_{i}=x_{i1}.
\end{align}

Then,  one can design $u_i$ as  
\begin{align}\label{epp1}
u_i=&-\phi_i(x_{i1}+\int_{0}^{t}\nabla_{i}f_i(\mathbf{z}_i(\tau))d\tau),
\end{align}
where the definitions of other variables follow those in \eqref{eq1}-\eqref{ada}. Following Theorem \ref{them1}, the subsequent corollary can be obtained. 

\begin{Corollary}\label{cor1}
	Under Assumptions \ref{ass1}-\ref{ass3} and the control input in  \eqref{epp1}
	\begin{align}
		\lim_{t\rightarrow \infty}||\mathbf{y}(t)-\mathbf{y}^*||=0,
	\end{align}
	and all the other variables stay bounded and converge to some finite values. 
\end{Corollary}
\begin{Proof}
	See Section \ref{proof_cor1} for the proof. 
\end{Proof}

 Moreover, for second-order players, the seeking strategy in \eqref{eq1} is written as
\begin{align}\label{eq444}
u_i=&-\theta_i \phi_i(x_{i2})\\
&-\theta_i^{2}\phi_i(\theta_ix_{i1}+x_{i2} +\theta_i\int_{0}^{t}\nabla_{i}f_i(\mathbf{z}_i(\tau))d\tau),\nonumber
\end{align}
with other variables defined in \eqref{eq1}.

Compared with bounded controls designed for first- and second-order players in \cite{YETAC2021}, the control inputs in \eqref{epp1} and \eqref{eq444} provide alternative designs to achieve distributed Nash equilibrium seeking with bounded controls. Moreover, the presented methods have the following elegant features:
\begin{enumerate}[i)]
	\item  The presented methods are fully distributed while the methods in \cite{YETAC2021} contain control gains depending on some global information. 
	\item  It is shown that the Nash equilibrium is globally asymptotically stable under the proposed methods, while in \cite{YETAC2021}, only semi-global results were given for second-order players.
	\item  The requirement on the boundedness of $\frac{\partial f_i(\mathbf{y})}{\partial y_i\partial y_j}$ for all $i,j\in\mathcal{V}$ in \cite{YETAC2021} is removed from the paper. 
\end{enumerate}

\subsection{Undirected communication graphs}
 In \cite{YEAUTO2021}, adaptive approaches are proposed to achieve fully distributed Nash equilibrium seeking for \textbf{first-order} players under undirected communication graphs.  For the case of undirected communication graph, following the adaptive design in \cite{YEAUTO2021}, $u_i$ is designed as 
 \begin{align}\label{eqpqp1}
 u_i=&-\sum\nolimits_{k=1}^{m_i-1}\theta_i^k \phi_i(\bar{x}_{i(m_i-k+1)})\allowdisplaybreaks\\
 &-\theta_i^{m_i}\phi_i(\bar{x}_{i1}+\prod_{k=1}^{m_i-1}\theta_i^k\int_{0}^{t}\nabla_{i}f_i(\mathbf{z}_i(\tau))d\tau),\nonumber\allowdisplaybreaks\\
 \dot{z}_{ij}=&-c_{ij}\xi_{ij},~~~~~~
 \dot{c}_{ij}=\xi_{ij}^2,\nonumber
 \end{align}
Correspondingly, the following corollary can be obtained.
 \begin{Corollary}\label{col_5}
 	Under Assumptions \ref{ass1}-\ref{ass2} and the control input in \eqref{eqpqp1},
 \begin{align}
 	\lim_{t\rightarrow \infty} ||\mathbf{y}(t)-\mathbf{y}^*||=0,
 \end{align}
  and all the other variables stay bounded given that the communication graph is undirected and connected. 
 \end{Corollary}
 \begin{Proof}
 	See Section \ref{proof_col_5} for the proof. 
 \end{Proof}

 Corollary \ref{col_5} indicates that under undirected communication graphs, the adaptive law in \cite{YEAUTO2021} can be employed to establish the control law for high-order players.  However, the analysis  depends on symmetric information exchange among the players and hence, the adaptive designs therein fail to work for directed communication graphs.  Therefore, this paper has the following advantages:
 \begin{enumerate}[i)]
 	\item  Unlike \cite{YEAUTO2021},\cite{Persis2019},\cite{Bianchi2021} that only consider undirected information exchange, the presented design in this paper can accommodate directed graphs.
 	\item Different from \cite{YEAUTO2021} that only considered first-order players,  players with multi-integrator type dynamics are addressed, which cover first-order ones as special cases.
 	\item The control inputs in this paper are restricted within a predefined domain, while in  \cite{YEAUTO2021},\cite{Persis2019},\cite{Bianchi2021}, the actuator limitations were not addressed. 
 	\item Different from \cite{Persis2019},\cite{Bianchi2021} that required two-hop communications, only one-hop communication is needed, which is desirable for distributed systems.   
 \end{enumerate}

 \subsection{Without actuator limitation}
 If the system is without any actuator limitation, the saturation function can be removed from the designed controls, which gives the following control input
 \begin{align}\label{actu_limi}
 u_i=&-\sum\nolimits_{k=1}^{m_i-1}\theta_i^k \bar{x}_{i(m_i-k+1)}\nonumber\\
 &-\theta_i^{m_i}(\bar{x}_{i1}+\prod_{k=1}^{m_i-1}\theta_i^k\int_{0}^{t}\nabla_{i}f_i(\mathbf{z}_i(\tau))d\tau),
 \end{align}
 with other variables defined in \eqref{eq1}-\eqref{ada}. 
 
In this case, the proposed method is still effective and the following corollary can be obtained. 
 \begin{Corollary}\label{cor41}
 	Under Assumptions \ref{ass1}-\ref{ass3} and \eqref{actu_limi},  
 	\begin{align}
 	\lim_{t\rightarrow \infty} ||\mathbf{y}(t)-\mathbf{y}^*||=0.   
 	\end{align} 
 	In addition, all the other variables stay bounded and converge to finite values. 
 \end{Corollary}
 \begin{Proof}
 	The proof can be completed by following Steps 2-3 in the proof of Theorem \ref{them1}. 
 \end{Proof}
 
From the above discussions, it is clear that the considered problem covers the problem addressed in \cite{YETAC2021} as a special case. Moreover, for undirected graphs, the adaptive design in \cite{YEAUTO2021} can be employed in the control design to find the Nash equilibrium in a fully distributed fashion. 
% From the above observations, one finds that the presented methods cover the problem considered in \cite{YETAC2021} as a special case, and the adaptive designs in  \cite{YEAUTO2021},\cite{Persis2019},\cite{Bianchi2021} fail to work for directed graphs. 
 \section{Numerical verifications}\label{num}
In this section, a numerical example with $6$ players is simulated. In the considered game, each player $i$'s objective function is defined as 
\begin{align}
	f_i(\mathbf{y})=&y_i^2+y_i+(y_i-y_{i+1})^2,  i\in\{1,2,\cdots,5\},\nonumber\\
	f_6(\mathbf{y})=&y_6^2+y_6+(y_6-y_1)^2,\nonumber
	\end{align}
by which the Nash equilibrium is $y_i=-0.5,\forall i\in\{1,2,\cdots,6\}$. 
In the simulation, the communication graph $\mathcal{G}$ is given in Fig. \ref{graph1}, which is directed and strongly connected. In addition, $\theta_i=\frac{1}{3}$, $m_i=3$,  $x_{i1}(0)=i$ and the initial conditions for all the other variables in \eqref{eq1}-\eqref{ada} are set as $1$. 
\begin{figure}
	\centering
	\includegraphics[width=5.5cm,height=3cm]{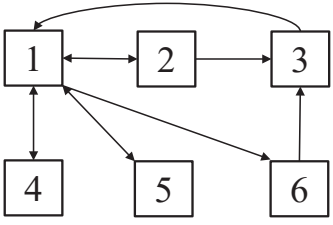}
	\caption{The strongly connected digraphs for the players.}\label{graph1}
\end{figure}

In the simulation, it is supposed that   $\Delta_i=1,$ which ensures that $|u_i|\leq0.4815$ for all $i\in\{1,2,\cdots,6\}$. With the control input design in 
\eqref{eq1}-\eqref{ada}, the evolution of the players' actions and control inputs are shown in Fig. \ref{figure1}, from which it is clear that the players' actions are convergent to the Nash equilibrium and the control inputs are restricted within the required domain. Moreover, the auxiliary  variables $c_{ij}(t)$ and $z_{ij}(t)$ are plotted in Fig. \ref{figure2}, from which it can be seen that they stay bounded and converge to finite values. 
To this end, the convergence of the developed algorithm has been numerically validated. 

\begin{figure}[t]
	\centering
	\includegraphics[width=8.5cm,height=5.5cm]{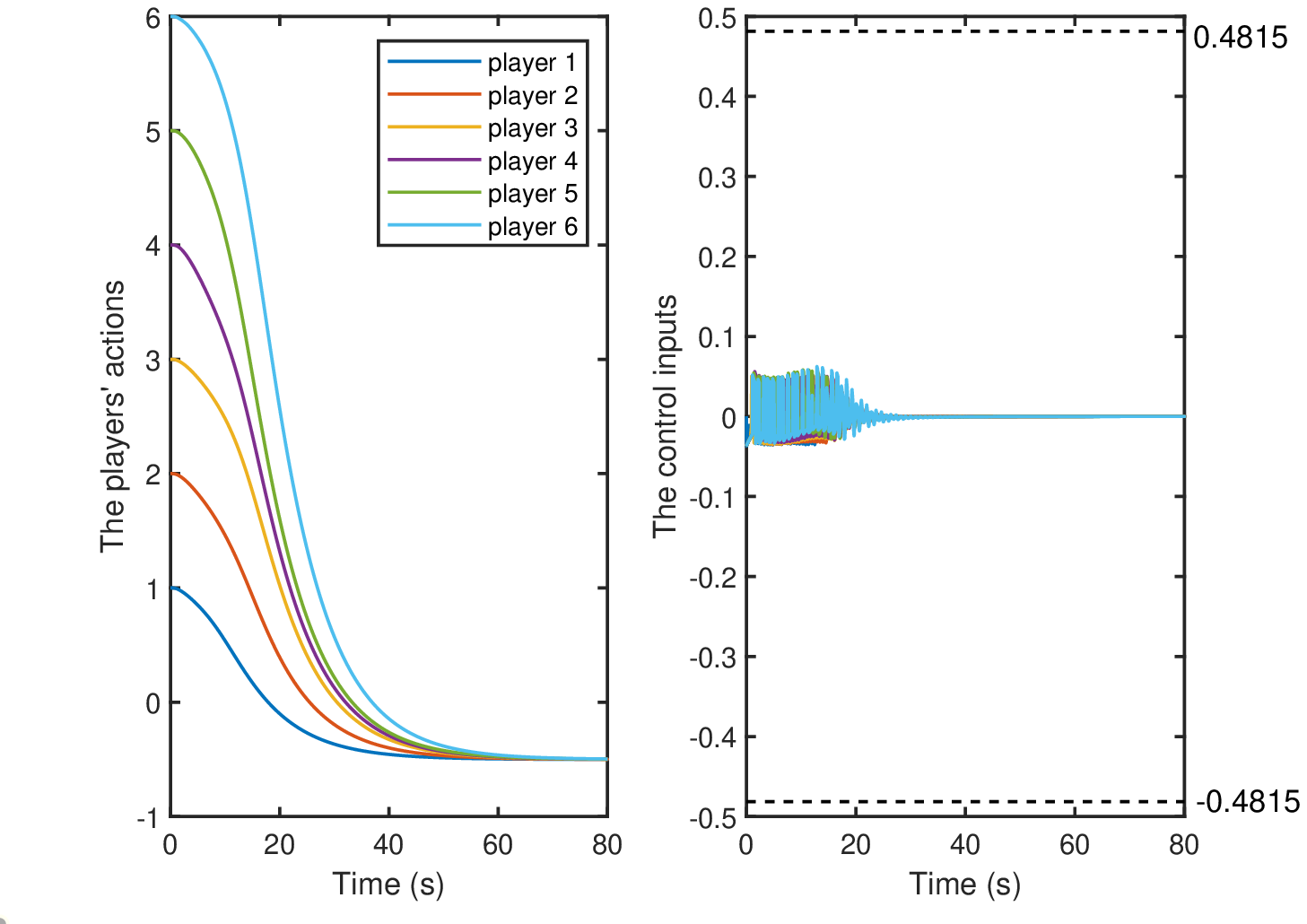}
	\caption{The players' actions $y_i(t)$ and control signals $u_i$ generated by \eqref{eq1}-\eqref{ada}.}\label{figure1}
\end{figure}

\begin{figure}[t]
	\centering
	\includegraphics[width=8.5cm,height=5.5cm]{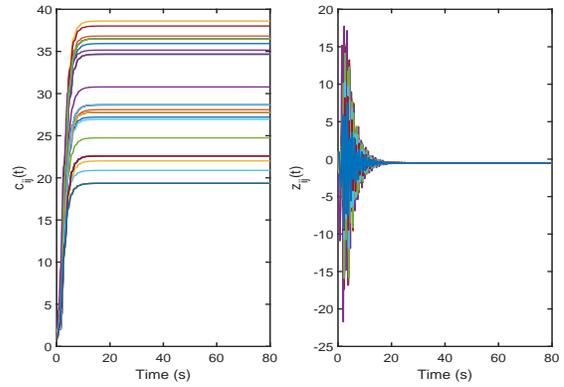}
	\caption{The evolution of the auxiliary variables $c_{ij}(t)$ and $z_{ij}(t)$ for $i,j\in\{1,2,\cdots,6\}$ generated by \eqref{eq1}-\eqref{ada}.}\label{figure2}
\end{figure}
\begin{figure}[t]
	
	\hspace{-2mm}\includegraphics[width=8.5cm,height=5.5cm]{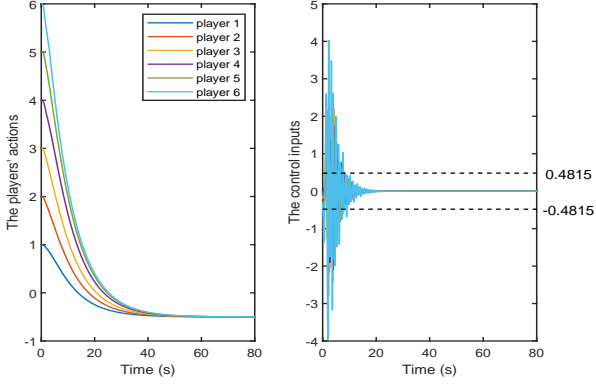}
	\caption{The players' actions and control inputs generated by \eqref{actu_limi}.}\label{figure3}
\end{figure}
To further illustrate the functionality of the  saturation functions in the proposed method, they are removed and correspondingly the method in  \eqref{actu_limi} is simulated. With all the settings kept the same as the case with saturation functions, the players' actions and control inputs generated by \eqref{actu_limi} are plotted in Fig.  \ref{figure3}. From this figure, it is clear that the players' actions are still convergent to the Nash equilibrium but the controls are sometimes out of $[-0.4815,0.4815].$ Comparing Fig. \ref{figure1} with Fig. \ref{figure3}, it can be concluded that the saturation functions are effective to restrict the controls within the required domain.

 \section{Conclusions}\label{concl}
 This paper contributes to finding the Nash equilibrium in a fully distributed fashion for high-order players subject to actuator limitations. A linear transformation is firstly applied to the players' dynamics, based on which multiple saturation functions are employed to develop the control inputs. With the saturation functions, the control inputs can be restricted within the required region. Moreover,  the control gains are designed to be adaptive, which allow asymmetric information exchange among the players and lead to fully distributed schemes. It is proven that, by the designed bounded control inputs, the players' actions are convergent to the Nash equilibrium.

\section{Appendix}
\subsection{Proof of Lemma \ref{Lemma1}}\label{proof_Lemma1}

By \eqref{eq_transf} and \eqref{eq1},
it can be obtained that 
\begin{align}
%\dot{\bar{x}}_{im_i}=&-\sum_{k=1}^{m_i-1}\theta_i^k \phi_i(\bar{x}_{i(m_i-k+1)}) \nonumber\\
%&-\theta_i^{m_i}\phi_i(\bar{x}_i+\prod_{k=1}^{m_i-1}\theta_i^k\int_{0}^{t}\nabla_{i}f_i(\mathbf{z}_i(\tau))d\tau)\nonumber\\
\dot{\bar{x}}_{im_i}=& -\theta_i \phi_i(\bar{x}_{im_i})-\sum\nolimits_{k=2}^{m_i-1}\theta_i^k \phi_i(\bar{x}_{i(m_i-k+1)}) \nonumber\\
&-\theta_i^{m_i}\phi_i(\bar{x}_i+\prod_{k=1}^{m_i-1}\theta_i^k\int_{0}^{t}\nabla_{i}f_i(\mathbf{z}_i(\tau))d\tau) \nonumber\\
\leq &-\theta_i \phi_i(\bar{x}_{im_i})+\theta_i^2\Delta_i/(1-\theta_i). \label{m_comp}
\end{align}

Define $V_{im_i}=\int_{0}^{\bar{x}_{im_i}}\phi_i(\tau)d\tau. $
%\begin{align}
%V_{im_i}=\int_{0}^{\bar{x}_{im_i}}\phi_i(\tau)d\tau. 
%\end{align}
Then, it can be easily obtained that
\begin{align}
V_{im_i}=\left\{
\begin{aligned}
&\Delta_i^2/2+(|x_{im_i}|-\Delta_i)\Delta_i~~~~~~~\text{if}~|x_{im_i}|>\Delta_i
\nonumber	\\&x_{im_i}^2/2~~~~~~~~~~~~~~~~~\text{if}~|x_{im_i}|\leq \Delta_i. \nonumber
\end{aligned}
\right.
\end{align}
Therefore, $V_{im_i}$ is positive definite and radially unbounded.  By Lemma 4.3 in \cite{Nonlinear2002}, 	there are $\mathcal{K}_{\infty}$ functions $\alpha_1$ and $\alpha_2$ such that 
$
\alpha_1(|\bar{x}_{im_i}|)\leq V_{im_i}\leq \alpha_2(|\bar{x}_{im_i}|). 
$ 
Taking the time derivative of $V_{im_i}$ gives
\begin{align}
\dot{V}_{im_i}\leq& -\theta_i \phi_i^2(\bar{x}_{im_i})+\theta_i^2\Delta_i|\phi_i(\bar{x}_{im_i})|/(1-\theta_i)\nonumber\\
\leq	&-\sigma_1 \theta_i \phi_i^2(\bar{x}_{im_i}), 
\end{align}
%\begin{align}
%|\phi_i(\bar{x}_{im_i})|>\frac{\theta_i\Delta_i}{(1-\sigma_1)(1-\theta_i)},
%\end{align}
for $|\phi_i(\bar{x}_{im_i})|>\theta_i\Delta_i/((1-\sigma_1)(1-\theta_i)),$
where $\sigma_1\in(0,1)$ is a constant.

\textbf{Case I:} $|\phi_i(\bar{x}_{im_i})|=\Delta_i$. Let $\theta_i^*$ be a positive constant such that $\frac{\theta_i^*}{(1-\sigma_1)(1-\theta_i^*)}=1$. Then,  it is clear that $\dot{V}_{im_i}<0$ is always satisfied for $\theta_i\in (0,\theta_i^*)$, indicating that if $|\bar{x}_{im_i}(0)|>\Delta_i$,  $|\bar{x}_{im_i}(t)|$ is bounded and  there exists a positive  constant $\bar{T}_1$ such that $|\bar{x}_{im_i}(t)|\leq\Delta_i$ for $t>\bar{T}_1.$ 

\textbf{Case II:} $|\phi_i(\bar{x}_{im_i})|=\bar{x}_{im_i}$. In this case, $\dot{V}_{im_i}\leq	-\sigma_1 \theta_i \phi_i^2(\bar{x}_{im_i}),$
%\begin{align}
%\dot{V}_{im_i}\leq	-\sigma_1 \theta_i \phi_i^2(\bar{x}_{im_i}),
%\end{align}
for all $|\bar{x}_{im_i}|>\frac{\theta_i\Delta_i}{(1-\sigma_1)(1-\theta_i)}.$
%\begin{align}
%|\bar{x}_{im_i}|>\frac{\theta_i\Delta_i}{(1-\sigma_1)(1-\theta_i)}.
%\end{align}
Note that as $V_{im_i}$ itself is a $\mathcal{K}_{\infty}$ function, one can choose $\alpha_1(|\bar{x}_{im_i}|)=\alpha_2(|\bar{x}_{im_i}|)=V_{im_i}$ and hence, if $\theta_i<\theta_i^*,$
$\frac{\theta_i}{(1-\sigma_1)(1-\theta_i)}<1$. Then,  there exists  a class $\mathcal{KL}$ function $\beta_1$ and for every $|\bar{x}_{im_i}(0)|<\Delta_i$, there exists a constant $\tilde{T}_1\geq 0$ such that 	
\begin{align}
|\bar{x}_{im_i}(t)|\leq & \beta_1(|\bar{x}_{im_i}(0)|,t),\forall t<\tilde{T}_1,\nonumber\\
|\bar{x}_{im_i}(t)|<&\theta_i\Delta_i/((1-\sigma_1)(1-\theta_i)),\forall t\geq\tilde{T}_1,
\end{align}
based on Theorem 4.19 in \cite{Nonlinear2002}.

Summarizing the above two cases, one gets that for any initial condition, $|\bar{x}_{im_i}(t)|\leq\Delta_i,\forall t\geq T_1,$
%\begin{align}
%|\bar{x}_{im_i}(t)|\leq\Delta_i,\forall t>T_1,
%\end{align}
For some $T_1\geq 0$. 
Note that for each $\theta_i\in (0,\frac{1}{2})$, $\frac{\theta_i}{(1-\sigma_1)(1-\theta_i)}<1$ is satisfied and hence, the above conclusion holds for all $\theta\in(0,\frac{1}{2}).$
Recalling that 
\begin{align}
\dot{\bar{x}}_{i(m_i-1)}=&\theta_i \bar{x}_{im_i}+u_i,\nonumber
\end{align}
it can be easily obtained that there is no finite escape time for $x_{i(m_i-1)}(t)$ based on the boundedness of $\bar{x}_{im_i}$ and the control inputs. Therefore, $x_{i(m_i-1)}(t)$  would stay bounded for $t<T_1$. Moreover, for $t\geq T_1$, one has
\begin{align}
&\dot{\bar{x}}_{i(m_i-1)}=\theta_i \bar{x}_{im_i}+u_i\nonumber\\
=	& -\theta_i^2\phi_i(\bar{x}_{i(m_i-1)})-\sum\nolimits_{k=3}^{m_i-1}\theta_i^k \phi_i(\bar{x}_{i(m_i-k+1)}) \nonumber\\
&-\theta_i^{m_i}\phi_i(\bar{x}_i+\prod_{k=1}^{m_i-1}\theta_i^k\int_{0}^{t}\nabla_{i}f_i(\mathbf{z}_i(\tau))d\tau).
\end{align}

Define $V_{i(m_i-1)}=\int_{0}^{\bar{x}_{i(m_i-1)}}\phi_i(\tau)d\tau.$
%\begin{align}
%V_{i(m_i-1)}=\int_{0}^{\bar{x}_{i(m_i-1)}}\phi_i(\tau)d\tau.\nonumber
%\end{align}
Then, it can be easily obtained that
\begin{align*}
\dot{V}_{i(m_i-1)}\leq -\theta_i^2 \phi_i^2(\bar{x}_{i(m_i-1)})+\theta_i^3\Delta_i|\phi_i(\bar{x}_{i(m_i-1)})|/(1-\theta_i).
\end{align*}

By similar analysis to that for $\bar{x}_{im_i}$, one gets that there exists a positive constant $T_2\geq T_1$ such that $|\bar{x}_{i(m_i-1)}(t)|\leq \Delta_i,\forall t<T_2,$ 
%\begin{align}
%|\bar{x}_{i(m_i-1)}(t)|\leq \Delta_i,\forall t<T_2,
%\end{align}
given that $\theta_i\in (0,\theta_i^*).$
Repeating the above process, it can be obtained that there exists a constant $T\geq 0$ such that if $t\geq T,$ $|\bar{x}_{ik}(t)|\leq \Delta_i, $
%\begin{align}
%|\bar{x}_{ik}(t)|\leq \Delta_i, 
%\end{align}
for all $k=2,\cdots,m_i.$	
\subsection{Proof of Lemma \ref{lemma2}}\label{proof_lemma2}
	As $|\nabla_{i}f_i(\mathbf{z}_i(t))|<\nu_1$ for all $t\geq 0$, one gets that 
\begin{align}	|\tilde{x}_{i1}(t)-\tilde{x}_{i1}(0)|&\leq \theta_i^{m_i}\Delta_i t+\prod\nolimits_{k=1}^{m_i-1}\theta_i^k\nu_1 t,
\end{align}
from \eqref{auxi_3} by utilizing the Comparison Lemma in \cite{Nonlinear2002}. 
Therefore, for any bounded $t$, $\tilde{x}_{i1}(t)$ is bounded and the system in \eqref{auxi_3} cannot have finite escape time.

The following analysis is conducted for $t\geq \tilde{T}.$
Define $V_{i1}=\int_{\tilde{T}}^{\tilde{x}_{i1}}\phi_i(\tau)d\tau.$
%\begin{align}
%V_{i1}=\int_{\tilde{T}}^{\tilde{x}_{i1}}\phi_i(\tau)d\tau.
%\end{align}
Then, for $t\geq \tilde{T},$
\begin{align}
\dot{V}_{i1}=&\phi_i(\tilde{x}_{i1})(-\theta_i^{m_i}\phi_i(\tilde{x}_{i1})+\prod\nolimits_{k=1}^{m_i-1}\theta_i^k\nabla_{i}f_i(\mathbf{z}_i(t)))\nonumber\\
\leq& -\theta_i^{m_i}\phi_i^2(\tilde{x}_{i1})/2,
\end{align}
for all $|\phi_i(\tilde{x}_{i1})|>\frac{2}{\theta_i^{m_i}}|\prod_{k=1}^{m_i-1}\theta_i^k\nabla_{i}f_i(\mathbf{z}_i(t))|.$

\textbf{Case I:} $|\phi_i(\tilde{x}_{i1})|=\Delta_i$. If this is the case, $\dot{V}_{i1}\leq -\frac{1}{2}\theta_i^{m_i}\phi_i^2(\tilde{x}_{i1}),$
%\begin{align}
%\dot{V}_{i1}\leq -\frac{1}{2}\theta_i^{m_i}\phi_i^2(\tilde{x}_{i1}),
%\end{align}
is always satisfied as for $t\geq \tilde{T}$, $\Delta_i>\frac{2}{\theta_i^{m_i}}|\prod_{k=1}^{m_i-1}\theta_i^k\nabla_{i}f_i(\mathbf{z}_i(t))|,$ indicating that for all $|\tilde{x}_{i1}(\tilde{T})|>\Delta_i,$ $|\tilde{x}_{i1}(t)|$ will evolve into the unsaturated region after some finite time. 

\textbf{Case II: $\phi_i(\tilde{x}_{i1})=\tilde{x}_{i1}$}. In this case, $\dot{V}_{i1}\leq -\frac{1}{2}\theta_i^{m_i}\phi_i^2(\tilde{x}_{i1}),$
%\begin{align}
%\dot{V}_{i1}\leq& -\frac{1}{2}\theta_i^{m_i}\phi_i^2(\tilde{x}_{i1}),
%\end{align}
for all $|\tilde{x}_{i1}|>\frac{2}{\theta_i^{m_i}}|\prod_{k=1}^{m_i-1}\theta_i^k\nabla_{i}f_i(\mathbf{z}_i(t))|.$
Therefore, by Theorem 4.18 in \cite{Nonlinear2002}, one gets that for $t\geq \tilde{T}$ 
\begin{align}
&|\tilde{x}_{i1}(t)|\leq  \beta(|\tilde{x}_{i1}(\tilde{T})|,t-\tilde{T})\nonumber\\
&+\alpha_1^{-1}(\alpha_2(\sup\nolimits_{\tilde{T}<\tau<t}2|\prod\nolimits_{k=1}^{m_i-1}\theta_i^k\nabla_{i}f_i(\mathbf{z}_i(\tau))|/\theta_i^{m_i}))\nonumber\\
\leq & \beta(|\tilde{x}_{i1}(\tilde{T})|,t-\tilde{T})+\gamma(\sup\nolimits_{\tilde{T}<\tau<t} |\nabla_{i}f_i(\mathbf{z}_i(\tau))|),\nonumber
\end{align}	
where $\gamma(\cdot)$ is a $\mathcal{K}_{\infty}$ function as  $\alpha_1(\cdot)$ and $\alpha_2(\cdot)$ are $\mathcal{K}_{\infty}$ functions (defined in the proof of Lemma \ref{Lemma1}) for all $|\tilde{x}_{i1}(\tilde{T})|<\Delta_i$. 
To this end, the conclusions have been obtained.
\subsection{Proof of Lemma \ref{lemma3}}\label{proof_lemma3}
To show the convergence property of \eqref{auxi_1},  let $V=V_1+V_2+V_3$ in which
\begin{align}
V_1=&\frac{1}{2}||[-\int_{0}^{t}\nabla_{i}f_i(\mathbf{z}_i(\tau))d\tau]_{vec}-\mathbf{y}^*||^2,\nonumber\allowdisplaybreaks\\
V_2=&\epsilon \sum\nolimits_{i=1}^N\sum\nolimits_{j=1}^N p_{ij}(c_{ij}+\rho_{ij}/2)\rho_{ij},\nonumber\allowdisplaybreaks\\
V_3=&\epsilon\sum\nolimits_{i=1}^N \sum\nolimits_{j=1}^Np_{ij}(c_{ij}-c^*)^2/2,
\end{align}
where $P=\text{diag}\{p_{ij}\}$ satisfies 
$PH+H^TP=Q$, $Q$ is a symmetric positive definite matrix as the communication graph is strongly connected, $\epsilon$ and $c^*$ are positive constants to be further quantified.
Then, 
\begin{align}\label{eq4}
\dot{V}_2=&\epsilon \sum_{i=1}^N\sum_{j=1}^N p_{ij}(c_{ij}+\frac{\rho_{ij}}{2})\dot{\rho}_{ij}+\epsilon \sum_{i=1}^N\sum_{j=1}^N p_{ij}(\dot{c}_{ij}+\frac{\dot{\rho}_{ij}}{2}+)\rho_{ij} \nonumber\\
%=&\epsilon \sum_{i=1}^N\sum_{j=1}^N p_{ij}(c_{ij}+\rho_{ij})\dot{\rho}_{ij}+\epsilon \sum_{i=1}^N\sum_{j=1}^N p_{ij}\dot{c}_{ij}\rho_{ij} \nonumber\\
=&\epsilon \sum_{i=1}^N\sum_{j=1}^N p_{ij}(c_{ij}+\rho_{ij})\dot{\rho}_{ij}+\epsilon \sum_{i=1}^N\sum_{j=1}^N p_{ij}\rho_{ij}^2.
\end{align}
In addition, 
\begin{align}\label{eq5}
\dot{V}_3
%=&\epsilon\sum\nolimits_{i=1}^N \sum\nolimits_{j=1}^Np_{ij}(c_{ij}-c^*)\dot{c}_{ij}\nonumber\\
=&\epsilon\sum\nolimits_{i=1}^N \sum\nolimits_{j=1}^Np_{ij}(c_{ij}-c^*)\rho_{ij}.
\end{align}
Combining \eqref{eq4}-\eqref{eq5}, one can derive that $\dot{V}_2+\dot{V}_3=\epsilon \sum_{i=1}^N\sum_{j=1}^N p_{ij}(c_{ij}+\rho_{ij})\dot{\rho}_{ij}+\epsilon\sum_{i=1}^N \sum_{j=1}^Np_{ij}(\rho_{ij}+c_{ij}-c^*)\rho_{ij},
$
%\begin{align}
%\dot{V}_2+\dot{V}_3=&\epsilon \sum_{i=1}^N\sum_{j=1}^N p_{ij}(c_{ij}+\rho_{ij})\dot{\rho}_{ij}\nonumber\\
%&+\epsilon\sum_{i=1}^N \sum_{j=1}^Np_{ij}(\rho_{ij}+c_{ij}-c^*)\rho_{ij},
%\end{align}
in which 
\begin{align}
&\epsilon \sum\nolimits_{i=1}^N\sum\nolimits_{j=1}^N p_{ij}(c_{ij}+\rho_{ij})\dot{\rho}_{ij} \nonumber\\
= &2\epsilon  \sum\nolimits_{i=1}^N\sum\nolimits_{j=1}^N p_{ij}(c_{ij}+\rho_{ij})\xi_{ij}\dot{\xi}_{ij}\nonumber\\
=&-\epsilon \mathbf{\xi}^T (c+\rho)(PH+H^TP)(c+\rho) \mathbf{\xi}\nonumber\\
&+2\epsilon \mathbf{\xi}^T (c+\rho)PH(\mathbf{1}_N\otimes[\nabla_{i}f_i(\mathbf{z}_i(t))]_{vec})\nonumber\\
\leq &-\epsilon \underline{\lambda}\mathbf{\xi}^T(c+\rho)(c+\rho)\mathbf{\xi}\nonumber\\
&+2\epsilon \mathbf{\xi}^T (c+\rho)PH(\mathbf{1}_N\otimes[\nabla_{i}f_i(\mathbf{z}_i(t))]_{vec}),
\end{align}
where $\underline{\lambda}$ is the minimum eigenvalue of $Q.$

Note that  $2\epsilon \mathbf{\xi}^T (c+\rho)PH(\mathbf{1}_N\otimes[\nabla_{i}f_i(\mathbf{z}_i(t))]_{vec})
\leq 2\epsilon ||\mathbf{\xi}^T(c+\rho)||||PH\mathbf{1}_N\otimes[\nabla_{i}f_i(\mathbf{z}_i)-\nabla_{i}f_i([-\int_{0}^{t}\nabla_{i}f_i(\mathbf{z}_i(\tau))d\tau]_{vec})]_{vec}|| +2\epsilon||\mathbf{\xi}^T(c+\rho)|| ||PH\mathbf{1}_N\otimes[\nabla_{i}f_i([-\int_{0}^{t}\nabla_{i}f_i(\mathbf{z}_i(\tau))d\tau]_{vec})-\nabla_{i}f_i(\mathbf{y}^*)]_{vec}||
\leq  \frac{\epsilon\underline{\lambda}}{4}\mathbf{\xi}^T(c+\rho)(c+\rho)\mathbf{\xi}+\epsilon_1 \frac{\epsilon^2\underline{\lambda}}{4}\mathbf{\xi}^T(c+\rho)(c+\rho)\mathbf{\xi}+\frac{4\epsilon N \max\{p_{ij}\}^2||H||^2\max\{l_i\}^2||H^{-1}||^2||\mathbf{\xi}||^2}{\underline{\lambda}}+\frac{4N^2 \max\{p_{ij}\}^2||H||^2\max\{l_i\}^2||[-\int_{0}^{t}\nabla_{i}f_i(\mathbf{z}_i(\tau))d\tau]_{vec}-\mathbf{y}^*||^2}{\underline{\lambda}\epsilon_1}.$
%\begin{align}
%&2\epsilon \mathbf{\xi}^T (c+\rho)PH(\mathbf{1}_N\otimes[\nabla_{i}f_i(\mathbf{z}_i(t))]_{vec})\nonumber\\
%\leq& 2\epsilon ||\mathbf{\xi}^T(c+\rho)||||PH\mathbf{1}_N\otimes[\nabla_{i}f_i(\mathbf{z}_i)-\nabla_{i}f_i([-\int_{0}^{t}\nabla_{i}f_i(\mathbf{z}_i(\tau))d\tau]_{vec})||\nonumber\nonumber\\ 
%& +2 ||PH\mathbf{1}_N\otimes[\nabla_{i}f_i([-\int_{0}^{t}\nabla_{i}f_i(\mathbf{z}_i(\tau))d\tau]_{vec})-\nabla_{i}f_i(\mathbf{y}^*)]_{vec}||\nonumber\\ 
%&\times \epsilon||\mathbf{\xi}^T(c+\rho)||\nonumber\\
%\leq & \frac{\epsilon\underline{\lambda}}{4}\mathbf{\xi}^T(c+\rho)(c+\rho)\mathbf{\xi}++\frac{\epsilon^2\underline{\lambda}}{4}\mathbf{\xi}^T(c+\rho)(c+\rho)\mathbf{\xi}\epsilon_1 \nonumber\\
%&+\frac{4\epsilon N \max\{p_{ij}\}^2||H||^2\max\{l_i\}^2||H^{-1}||^2||\mathbf{\xi}||^2}{\underline{\lambda}}\nonumber\\
%&+\frac{4N^2 \max\{p_{ij}\}^2||H||^2\max\{l_i\}^2||[-\int_{0}^{t}\nabla_{i}f_i(\mathbf{z}_i(\tau))d\tau]_{vec}-\mathbf{y}^*||^2}{\underline{\lambda}\epsilon_1}.
%\end{align}

Moreover,  $\epsilon\sum_{i=1}^N \sum_{j=1}^Np_{ij}(\rho_{ij}+c_{ij}-c^*)\rho_{ij}
\leq \frac{\epsilon\underline{\lambda}}{4}\mathbf{\xi}^T(\rho+c)(\rho+c)\mathbf{\xi}-(\epsilon \min\{p_{ij}\}c^*-\frac{\max\{p_{ij}^2\}\epsilon}{\underline{\lambda}})||\mathbf{\xi}||^2$.
%\begin{align}
%&\epsilon\sum_{i=1}^N \sum_{j=1}^Np_{ij}(\rho_{ij}+c_{ij}-c^*)\rho_{ij}\nonumber\\
%\leq &\frac{\epsilon\underline{\lambda}}{4}\mathbf{\xi}^T(\rho+c)(\rho+c)\mathbf{\xi}\nonumber\\
%&-(\epsilon \min\{p_{ij}\}c^*-\frac{\max\{p_{ij}^2\}\epsilon}{\underline{\lambda}})||\mathbf{\xi}||^2
%\end{align}
Summarizing the above inequalities, one can derive that 
\begin{align}
&\dot{V}_2+\dot{V}_3\leq  
-(\epsilon \underline{\lambda}/2-\epsilon^2\epsilon_1\underline{\lambda}/4)\mathbf{\xi}^T(c+\rho)(c+\rho)\mathbf{\xi}\nonumber\\
&+p_1 ||[-\int_{0}^{t}\nabla_{i}f_i(\mathbf{z}_i(\tau))d\tau]_{vec}-\mathbf{y}^*||^2-p_2||\mathbf{\xi}||^2,
\end{align}
where $p_1=4N^2 \max\{p_{ij}\}^2||H||^2\max\{l_i\}^2/(\underline{\lambda}\epsilon_1)$ and $p_2=\epsilon \min\{p_{ij}\}c^*-\max\{p_{ij}^2\}\epsilon/\underline{\lambda}-4\epsilon N \max\{p_{ij}\}^2||H||^2\max\{l_i\}^2||H^{-1}||^2/{\underline{\lambda}}$.
Furthermore,  
\begin{align}
&\dot{V}_1=-\mathbf{r}^T [\nabla_if_i(\mathbf{z}_i)]_{vec}\nonumber\allowdisplaybreaks\\
&=-\mathbf{r}^T [\nabla_if_i([-\int_{0}^{t}\nabla_{i}f_i(\mathbf{z}_i(\tau))d\tau]_{vec})]_{vec}\nonumber\allowdisplaybreaks\\
&	-\mathbf{r}^T[\nabla_if_i(\mathbf{z}_i)-\nabla_if_i([-\int_{0}^{t}\nabla_{i}f_i(\mathbf{z}_i(\tau))d\tau]_{vec})]_{vec}\leq\nonumber\allowdisplaybreaks\\
&-\omega||\mathbf{r}||^2+\max\{l_i\}||\mathbf{z}+\mathbf{1}_N\otimes [\int_{0}^{t}\nabla_{i}f_i(\mathbf{z}_i(\tau))d\tau]_{vec}||||\mathbf{r}||\nonumber\allowdisplaybreaks\\
&\leq-\omega||\mathbf{r}||^2+\max\{l_i\}||H^{-1}||||\mathbf{\xi}||||\mathbf{r}||\nonumber\allowdisplaybreaks\\
&\leq -(\omega-\frac{ \max\{l_i\}||H^{-1}||}{2\epsilon_1})||\mathbf{r}||^2+\frac{\max\{l_i\}||H^{-1}||\epsilon_1}{2}||\mathbf{\xi}||^2,	\nonumber
\end{align}
where $\mathbf{r}=[-\int_{0}^{t}\nabla_{i}f_i(\mathbf{z}_i(\tau))d\tau]_{vec}-\mathbf{y}^*$ is defined for notational convenience. Therefore,
\begin{align}
&\dot{V}	\leq-(\epsilon \underline{\lambda}/2-\epsilon^2\epsilon_1\underline{\lambda}/4)\mathbf{\xi}^T(c+\rho)(c+\rho)\mathbf{\xi}\nonumber\\
&-(\epsilon \min\{p_{ij}\}c^*-\max\{p_{ij}^2\}\epsilon/\underline{\lambda}-\max\{l_i\}||H^{-1}||\epsilon_1/2-\nonumber\\
&4\epsilon N \max\{p_{ij}\}^2||H||^2\max\{l_i\}^2||H^{-1}||^2/\underline{\lambda})||\mathbf{\xi}||^2-p_3||\mathbf{r}||^2,\nonumber
\end{align}
where $p_3=\omega-\frac{4N^2 \max\{p_{ij}\}^2||H||^2\max\{l_i\}^2}{\underline{\lambda}\epsilon_1}-\frac{ \max\{l_i\}||H^{-1}||}{2\epsilon_1}.$

Choose $\epsilon_1$ such that $\epsilon_1>\frac{ \max\{l_i\}||H^{-1}||}{2\omega}+\frac{4N^2 \max\{p_{ij}\}^2||H||^2\max\{l_i\}^2}{\underline{\lambda}\omega},$
%\begin{align}
%\epsilon_1>\frac{ \max\{l_i\}||H^{-1}||}{2\omega}+\frac{4N^2 \max\{p_{ij}\}^2||H||^2\max\{l_i\}^2}{\underline{\lambda}\omega},\nonumber
%\end{align}
and  $\epsilon< \frac{2}{\epsilon_1}$.
In addition, $c^*>\frac{\max\{p_{ij}^2\}\epsilon}{\underline{\lambda}\epsilon \min\{p_{ij}\}}+\frac{\max\{l_i\}||H^{-1}||\epsilon_1}{2\epsilon \min\{p_{ij}\}}+\frac{4\epsilon N \max\{p_{ij}\}^2||H||^2\max\{l_i\}^2||H^{-1}||^2}{\underline{\lambda}\epsilon \min\{p_{ij}\}}.$
%\begin{align}
%c^*>&\frac{\max\{p_{ij}^2\}\epsilon}{\underline{\lambda}\epsilon \min\{p_{ij}\}}+\frac{\max\{l_i\}||H^{-1}||\epsilon_1}{2\epsilon \min\{p_{ij}\}}\nonumber\\
%&+\frac{4\epsilon N \max\{p_{ij}\}^2||H||^2\max\{l_i\}^2||H^{-1}||^2}{\underline{\lambda}\epsilon \min\{p_{ij}\}}.
%\end{align}
Then, $\dot{V}\leq 0$
%\begin{align}
%\dot{V}\leq 0
%\end{align}
and $V$ is bounded so as $[-\int_{0}^{t}\nabla_{i}f_i(\mathbf{z}_i(\tau))d\tau]_{vec}$, $\xi_{ij}$ and $c_{ij}$. In addition, for $\dot{V}=0$, $||\mathbf{\xi}||=0,$ and $||-[\int_{0}^{t}\nabla_{i}f_i(\mathbf{z}_i(\tau))d\tau]_{vec}-\mathbf{y}^*||=0.$
By further recalling the definition of $c_{ij}$, one can obtain that it is monotonically increasing, and hence it converges to some finite value as it is bounded.

\subsection{Proof of Theorem \ref{them1}}\label{proof_them1}
The proof  can be completed by several steps.

\textbf{Step 1: Analyze the evolution of the system for $t\leq T$ and $t>T$, respectively.}
According to Lemmas \ref{Lemma1}-\ref{lemma3}, there is no finite escape time for $\bar{x}_{ik}$, $z_{ij}$ and $c_{ij}$ where $i,j\in\mathcal{V}$ and $k\in\{1,2,\cdots,m_i\}$, indicating that for $t<T$, $\bar{x}_{ik}(t)$, $z_{ij}(t)$ and $c_{ij}(t)$ are all bounded. Moreover, by Lemma \ref{Lemma1}, it can be obtained that for $t>T$,
\begin{align}\label{step_1}
\dot{\bar{x}}_{i1}=&-\theta_i^{m_i}\phi_i(\bar{x}_{i1}+\prod_{k=1}^{m_i-1}\theta_i^k\int_{0}^{t}\nabla_{i}f_i(\mathbf{z}_i(\tau))d\tau),\nonumber\\
\dot{z}_{ij}=&-(c_{ij}+\rho_{ij})\xi_{ij},~~~~~\dot{c}_{ij}=\rho_{ij},
\end{align}	
where $\rho_{ij}=\xi_{ij}^2.$

\textbf{Step 2: Analyze the evolution of $\tilde{x}_{i1}$ for $t\rightarrow \infty$.}
By Lemma \ref{lemma3},	$\lim_{t\rightarrow \infty}||-[\int_{0}^{t}\nabla_{j}f_j(\mathbf{z}_j(\tau))d\tau]_{vec}-\mathbf{y}^*||=0,$
%\begin{align}
%\lim_{t\rightarrow \infty}||-[\int_{0}^{t}\nabla_{j}f_j(\mathbf{z}_j(\tau))d\tau]_{vec}-\mathbf{y}^*||=0,
%\end{align} 
and hence, by Barbarlat's Lemma \cite{Nonlinear2002}, one gets that $\lim_{t\rightarrow \infty} \nabla_{j}f_j(\mathbf{z}_j(t))= 0,$
%\begin{align}
%\lim_{t\rightarrow \infty} \nabla_{j}f_j(\mathbf{z}_j(t))= 0,
%\end{align}
indicating that there exists a positive constant $T_1>T$ such that for all $t>T_1,$
\begin{align}
|\tilde{x}_{i1}(t)|\leq \beta(|\tilde{x}_{i1}(T_1)|,t-T_1)+\gamma(\sup_{T_1<\tau<t} |\nabla_{i}f_i(\mathbf{z}_i(\tau))|),\nonumber
\end{align}
by Lemma \ref{lemma3}. Recalling that $\lim_{t\rightarrow \infty} \nabla_{j}f_j(\mathbf{z}_j(t))= 0,$ it is clear that  
$
\lim_{t\rightarrow \infty} |\tilde{x}_{i1}(t)|=0.
$

\textbf{Step 3: Analyze the steady state of $\bar{x}_{ik}$ for $k\in\{2,\cdots,m_i\}$.}
Recalling the dynamics in \eqref{eq_transf}, it can be obtained that for $t>T$,
\begin{align}\label{eqqq}
\dot{\bar{x}}_{i2}=-\theta_i^{m_i-1}\bar{x}_{i2}-\theta_i^{m_i}\phi_i(\tilde{x}_{i1}).
\end{align}
Regard $v_{im_i}=\theta_i^{m_i}\phi_i(\tilde{x}_{i1})$ as a virtual control input. Then, it can be easily obtained that the system in \eqref{eqqq} is input-to-state stable by defining a Laypunov candidate function as $\bar{V}=\frac{1}{2}\bar{x}_{i2}^2.$
As for $t\rightarrow \infty,$ $|v_{im_i}(t)|$ vanishes to zero, one gets that $\lim_{t\rightarrow \infty} |\bar{x}_{i2}(t)|=0.$
 Moreover, for $t>T,$
\begin{align}\label{eqq1q}
\dot{\bar{x}}_{i3}=-\theta_i^{m_i-2}\bar{x}_{i3}-\theta_i^{m_i-1}\bar{x}_{i2}-\theta_i^{m_i}\phi_i(\tilde{x}_{i1}).
\end{align}
Let $v_{i(m_i-1)}=-\theta_i^{m_i-1}\bar{x}_{i2}-\theta_i^{m_i}\phi_i(\tilde{x}_{i1})$ be the virtual control input, then, it can be easily obtained that \eqref{eqq1q} is input-to-state stable. Noticing that $\lim_{t\rightarrow \infty} |v_{i(m_i-1)}(t)|=0,$ one gets that  
$
\lim_{t\rightarrow \infty} |\bar{x}_{i3}(t)|=0.
$ 

Repeating the above process, one gets that 
\begin{align*}
\lim_{t\rightarrow \infty} |\bar{x}_{ik}(t)|=0,\forall k\in\{2,\cdots,m_i\}.
\end{align*}

\textbf{Step 4: Analyze the steady state of $\mathbf{y}(t)$.}
Recalling that $x_i=T_i^{-1}\bar{x}_i,$ 
and $y_i=x_{i1},$ one can obtain that 
\begin{align}
y_i=\bar{x}_{i1}/(\prod\nolimits_{k=1}^{m_i-1}\theta_i^k)+\sum\nolimits_{k=2}^{m_i}g_k(\theta_i)\bar{x}_{ik},
\end{align}
where $g_k(\theta_i)$ denotes some function of $\theta_i$.

Note that  by Lemma \ref{lemma3}, $\lim_{t\rightarrow \infty}||-[\int_{0}^{t}\nabla_{i}f_i(\mathbf{z}_i(\tau))d\tau]_{vec}-\mathbf{y}^*||=0,$
%\begin{align}
%\lim_{t\rightarrow \infty}||-[\int_{0}^{t}\nabla_{i}f_i(\mathbf{z}_i(\tau))d\tau]_{vec}-\mathbf{y}^*||=0,
%\end{align}
and $\lim_{t\rightarrow \infty} ||\mathbf{z}(t)+\mathbf{1}_N\otimes [\int_{0}^{t}\nabla_{i}f_i(\mathbf{z}_i(\tau))d\tau]_{vec}||=0,$
%\begin{align}
%\lim_{t\rightarrow \infty} ||\mathbf{z}(t)+\mathbf{1}_N\otimes [\int_{0}^{t}\nabla_{i}f_i(\mathbf{z}_i(\tau))d\tau]_{vec}||=0,
%\end{align}
then it is clear that 
$
\lim_{t\rightarrow \infty} ||\mathbf{y}(t)-\mathbf{y}^*||=0,
$
by further noticing that  $\lim_{t\rightarrow \infty} y_i(t)= \bar{x}_{i1}(t)/(\prod_{k=1}^{m_i-1}\theta_i^k),$  and $\lim_{t\rightarrow \infty} \bar{x}_{i1}(t)=\prod_{k=1}^{m_i-1}\theta_i^k y_i^*$ for all $i\in\mathcal{V}$.  To this end, the conclusions are apparent.
\subsection{Proof of Corollary \ref{cor1}}\label{proof_cor1}
In this case, 
\begin{align}\label{step_11}
\dot{x}_{i1}=&-\phi_i(x_{i1}+\int_{0}^{t}\nabla_{i}f_i(\mathbf{z}_i(\tau))d\tau),\nonumber\\
\dot{z}_{ij}=&-(c_{ij}+\rho_{ij})\xi_{ij},~~~~~\dot{c}_{ij}=\rho_{ij},
\end{align}	
where $\rho_{ij}=\xi_{ij}^2.$
Following Step 2 in the proof of Theorem \ref{them1}, $\lim_{t\rightarrow \infty} |\tilde{x}_{i1}(t)|=0,$
%\begin{align}
%\lim_{t\rightarrow \infty} |\tilde{x}_{i1}(t)|=0,
%\end{align} 
in which $\tilde{x}_{i1}(t)=x_{i1}+\int_{0}^{t}\nabla_{i}f_i(\mathbf{z}_i(\tau))d\tau$ in this case.
Moreover, by Lemma \ref{lemma3},  $\lim_{t\rightarrow \infty}||[-\int_{0}^{t}\nabla_{i}f_i(\mathbf{z}_i(\tau))d\tau]_{vec}-\mathbf{y}^*||=0,$
%\begin{align}
%\lim_{t\rightarrow \infty}||[-\int_{0}^{t}\nabla_{i}f_i(\mathbf{z}_i(\tau))d\tau]_{vec}-\mathbf{y}^*||=0,
%\end{align}
and hence 
$
\lim_{t\rightarrow \infty}||\mathbf{y}(t)-\mathbf{y}^*||=0,
$
%\begin{align*}
%\lim_{t\rightarrow \infty}||\mathbf{y}(t)-\mathbf{y}^*||=0,
%\end{align*}
from which the conclusions can be easily obtained and thus, the rest of the proof is omitted. 

\subsection{Proof of Corollary \ref{col_5}}\label{proof_col_5}
To prove the result, define an auxiliary system as 
\begin{align}
\dot{z}_{ij}=-c_{ij}\xi_{ij}~~~~\dot{c}_{ij}=\xi_{ij}^2.\label{auxi_6}
\end{align}
Define $
V=\frac{1}{2}||-[\int_{0}^{t}\nabla_{j}f_j(\mathbf{z}_j(\tau))d\tau]_{vec}-\mathbf{y}^*||^2+(\mathbf{z}+\mathbf{1}_N\otimes[\int_{0}^{t}\nabla_{j}f_j(\mathbf{z}_j(\tau))d\tau]_{vec})^TH(\mathbf{z}+\mathbf{1}_N\otimes[\int_{0}^{t}\nabla_{j}f_j(\mathbf{z}_j(\tau))d\tau]_{vec}) +\sum_{i=1}^{N}\sum_{j=1}^{N}(c_{ij}-c_{ij}^*)^2.$
%\begin{align}
%&V=\frac{1}{2}||-[\int_{0}^{t}\nabla_{j}f_j(\mathbf{z}_j(\tau))d\tau]_{vec}-\mathbf{y}^*||^2+\nonumber\\
%&(\mathbf{z}+\mathbf{1}_N\otimes[\int_{0}^{t}\nabla_{j}f_j(\mathbf{z}_j(\tau))d\tau]_{vec})^TH\times\nonumber\\
%&(\mathbf{z}+\mathbf{1}_N\otimes[\int_{0}^{t}\nabla_{j}f_j(\mathbf{z}_j(\tau))d\tau]_{vec}) +\sum_{i=1}^{N}\sum_{j=1}^{N}(c_{ij}-c_{ij}^*)^2.\nonumber
%\end{align}
Then, following the proof of Lemma \ref{lemma3} and \cite{YEAUTO2021},  one gets that 
 $\lim_{t\rightarrow\infty} ||-[\int_{0}^{t}\nabla_{j}f_j(\mathbf{z}_j(\tau))d\tau]_{vec}-\mathbf{y}^*||=0,$
%\begin{align}
%\lim_{t\rightarrow\infty} ||-[\int_{0}^{t}\nabla_{j}f_j(\mathbf{z}_j(\tau))d\tau]_{vec}-\mathbf{y}^*||=0,
%\end{align}
and
 $\lim_{t\rightarrow\infty}||\mathbf{z}+\mathbf{1}_N\otimes[\int_{0}^{t}\nabla_{j}f_j(\mathbf{z}_j(\tau))d\tau]_{vec}||=0$
%\begin{align}
%\lim_{t\rightarrow\infty}||\mathbf{z}+\mathbf{1}_N\otimes[\int_{0}^{t}\nabla_{j}f_j(\mathbf{z}_j(\tau))d\tau]_{vec}||=0,
%\end{align}
for \eqref{auxi_6}. 
The rest of the proof follows those in Theorem \ref{them1} and is omitted. 

\begin{thebibliography}{99}
\bibitem{YETAC2017} M. Ye and G. Hu, ``Distributed Nash equilibrium seeking by a consensus based approach," \emph{IEEE Transactions on Automatic
		Control}, vol. 62, no. 9, pp. 4811-4818, 2017.
\bibitem{wangauto2020}X. Wang, X. Sun, A. Teel and K. Liu, ``Distributed robust Nash
equilibrium seeking for aggregative games under persistent attacks: a hybrid systems approach," \emph{Automatica}, 122, 109255, 2020.
\bibitem{YEAUTO2021} M. Ye and G. Hu, ``Adaptive approaches for fully distributed Nash equilibrium seeking in networked games," \emph{Automatica}, vol. 129, no. 3, 109661, 2021.
\bibitem{ZhengCCC2020}Z. Zheng, Y. Zhang, B. Zhang, R. Yin, ``Distributed Nash equilibrium seeking of aggregative games for high-order systems," \emph{Chinese Control Conference}, pp. 4789-4794, 2020.
\bibitem{LiuICCA2020}X. Liu, Y. Zhang, X. Wang, and H. Ji, ``Distributed Nash equilibrium seeking design in network of uncertain linear multi-agent systems," \emph{IEEE International Conference on Control and Automation,} pp. 147-152, 2020.
\bibitem{RomanoTCSN20}A. R. Romano, L. Pavel, ``Dynamic NE seeking for multi-integrator networked agents with disturbance rejection," \emph{IEEE Transactions on Control of Network Systems}, vol. 7, no. 1, pp. 129-139, 2020.
\bibitem{YETAC2021} M. Ye, ``Distributed Nash equilibrium seeking for games in
systems with bounded control inputs," \emph{IEEE Transactions on Automatic Control}, vol. 66, no. 8, pp. 3833-3839, 2021. 
\bibitem{Persis2019}C. De Persis and S. Grammatico, ``Distributed averaging integral Nash equilibrium seeking on networks," \emph{Automatica}, vol. 110, pp. 108548, 2019.
\bibitem{Bianchi2021}M. Bianchi, S. Grammatico, ``Continuous-time fully distributed
generalized Nash equilibrium seeking for multi-integrator agents,"  \emph{Automatica,} vol. 129, 109660, 2021.
\bibitem{AiIJRNC2021}X. Ai and L. Wang, ``Distributed adaptive Nash equilibrium seeking and
disturbance rejection for noncooperative games of
high-order nonlinear systems with input saturation and input delay," \emph{International Journal of Robust and Nonlinear Control}, vol. 31, pp. 2827-2846, 2021.
\bibitem{YETAC183}M. Ye, G. Hu, L. Xie and S. Xu, ``Differentially private distributed Nash equilibrium seeking for aggregative games," \emph{IEEE Transactions on Automatic Control}, published online, DOI: 10.1109/TAC.2021.3075183.
\bibitem{LuTcyber2019}
K. Lu and Q. Zhu, ``Nonsmooth Continuous-Time Distributed Algorithms for Seeking Generalized Nash Equilibria of Noncooperative Games via Digraphs", \emph{IEEE Transactions on Cybernetics}, published online, DOI: 
10.1109/TCYB.2021.3049463.

\bibitem{KoshalOR2016}J. Koshal, A. Nedic and U. Shanbhag, ``Distributed algorithms for
aggregative games on graphs," \emph{Operations Research}, vol. 64, pp. 680-
704, 2016.
\bibitem{SalehisadaghianiAuto16} F. Salehisadaghiani and L. Pavel, ``Distributed Nash equilibrium seeking:
A gossip-based algorithm," \emph{Automatica}, vol. 72, pp. 209-216,
2016.
\bibitem{QinTAC}J. Qin, Q. Ma, W. X. Zheng, H. Gao, and Y. Kang, ``Robust $H_{\infty}$ group consensus for interacting clusters of  
integrator agents," \emph{IEEE Transactions on Automatic Control,} vol. 62, no. 7, pp. 3559-3566, 2017. 
\bibitem{DongCNS14}X. Dong, J. Xi, G. Lu, and Y. Zhong, ``Formation control for high-order linear time-invariant multiagent systems with time delays," \emph{IEEE Transactions on Control of Network Systems}, vol. 1, no. 3, pp. 232-240, 2014.
\bibitem{ShaoTAC20}J. Shao, W. X. Zheng, L. Shi, Y. Cheng, ``Leader-follower Cucker-Smale flocking with lossy links and general weight functions," \emph{IEEE Transactions on Automatic Control}, published online, DOI: 10.1109/TAC.2020.3046695, 2020.

\bibitem{Nonlinear2002}H. K. Khalil, \emph{Nonlinear Systems,}  Upper Saddle River, NJ:
Prentice Hall, 2002.
\bibitem{WangTCSII29}X. Li, Z. Sun, Y. Tang and H. R. Karimi, ``Adaptive event-triggered consensus of multi-agent systems on directed graphs," \emph{IEEE Transactions on Automatic Control,} vol. 66, no. 4, pp. 1670-1685, 2021.
\bibitem{SussmannTAC94} H. J. Sussmann, E. D. Sontag and Y. Yang, ``A general result on stabilization of linear systems using bounded controls," \emph{IEEE Transactions on Automatic Control,} vol. 39, no. 12, pp. 2411-2425, 1994.


\end{thebibliography}
\end{document}